\documentclass{amsart}
\usepackage{graphicx}
\usepackage{amssymb} 
\usepackage{amsmath} 
\usepackage{amsthm}
\usepackage{times}

\newtheorem{theorem}{Theorem}
\newtheorem{lemma}[theorem]{Lemma}

\newtheorem{definition}{Definition}
\newtheorem{algorithm}{Algorithm}

\newcommand{\bbN}{\mathbb{N}}
\newcommand{\bbC}{\mathbb{C}}
\newcommand{\bbZ}{\mathbb{Z}}

\author{Cyndie Cottrell and Benjamin Young}

\title{Domino shuffling for the Del Pezzo 3 lattice}

\begin{document}

\begin{abstract} 

We present a version of the domino shuffling algorithm (due to Elkies,
Kuperberg, Larsen and Propp) which works on a different lattice: the hexagonal
lattice superimposed on its dual graph.  We use our algorithm to count perfect matchings on a family of finite subgraphs
of this lattice whose boundary conditions are compatible with our algorithm.
In particular, we re-prove an enumerative theorem of Ciucu, as well as finding a
related family of subgraphs which have $2^{(n+1)^2}$ perfect matchings.  We
also give three-variable generating functions for perfect matchings on both
families of graphs, which encode certain statistics on the height functions of
these graphs.  
\end{abstract}

\maketitle

\begin{figure}
\caption{The $dP_3$ Lattice
\label{fig:the dp3 lattice}}
\begin{center}
\includegraphics[height=2in]{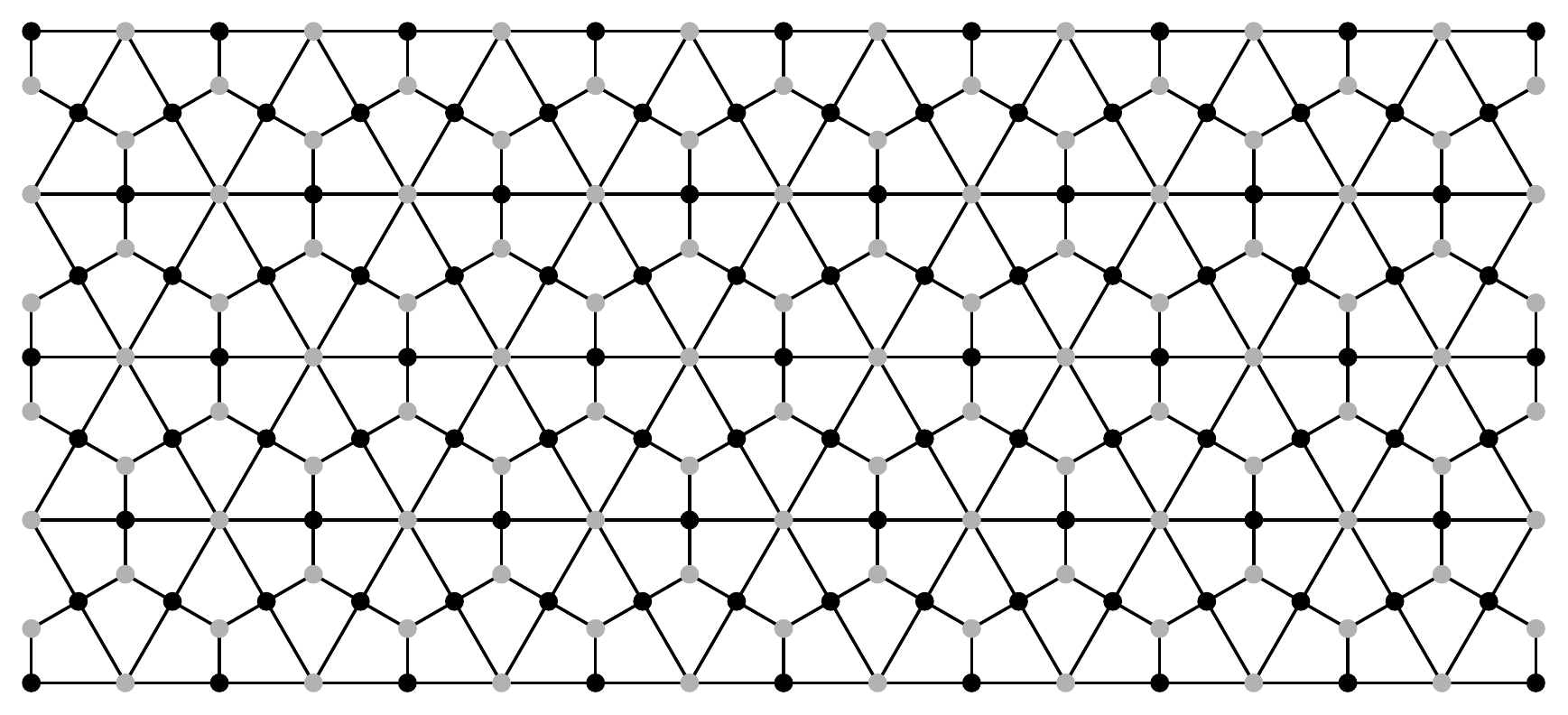}
\end{center}
\end{figure}

\section{Introduction}

The theory of perfect matchings on planar bipartite graphs is quite rich and
mature, and has seen a great deal of activity over the past two decades. The
central questions in this theory are counting questions, in which one attempts
to give a generating function, or even just a count, of perfect matchings on a
fixed planar bipartite graph G.

In the 1960s, Kasteleyn~\cite{kasteleyn} developed a powerful
algebraic tool for answering (in principle) all such questions. In particular,
one can compute the number of perfect matchings of a planar bipartite graph by
taking the determinant of a signed version of its bipartite adjacency matrix (usually called the \emph{Kasteleyn matrix}).

This paper is concerned with a different phenomenon which is not immediately explained by Kasteleyn's formalism.  Given a
periodic planar bipartite graph, it is often possible to find a family of
finite subgraphs for which the number of perfect matchings decomposes into
small, simple factors (as do the generating functions of the graphs). The
canonical example of this phenomenon is the Aztec Diamond of order $n$~\cite{eklp} , which
is a diamond-shaped region of the square lattice. The
Aztec diamond of order $n$ has $n(n+1)/2$ perfect matchings~\cite{eklp}.
Elkies et al. prove this by using a technique called \emph{domino shuffling}, which is
a type of random mapping on the set of perfect matchings on a square lattice.

\begin{figure}
\caption{Diamonds of orders 3 and 3.5
\label{fig:some diamonds}}
\begin{center}
\includegraphics[height=3in]{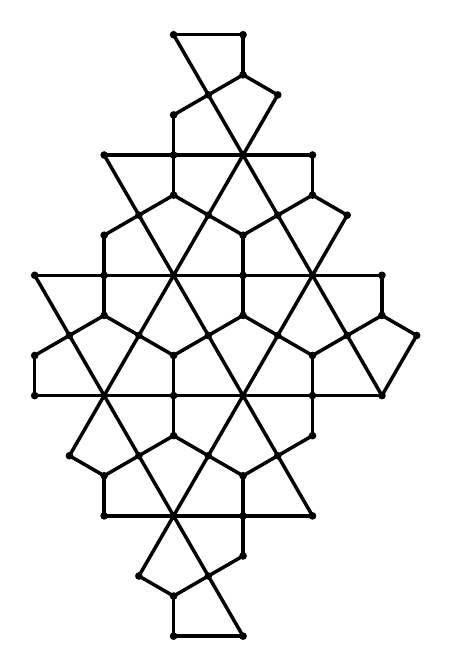}
\includegraphics[height=3in]{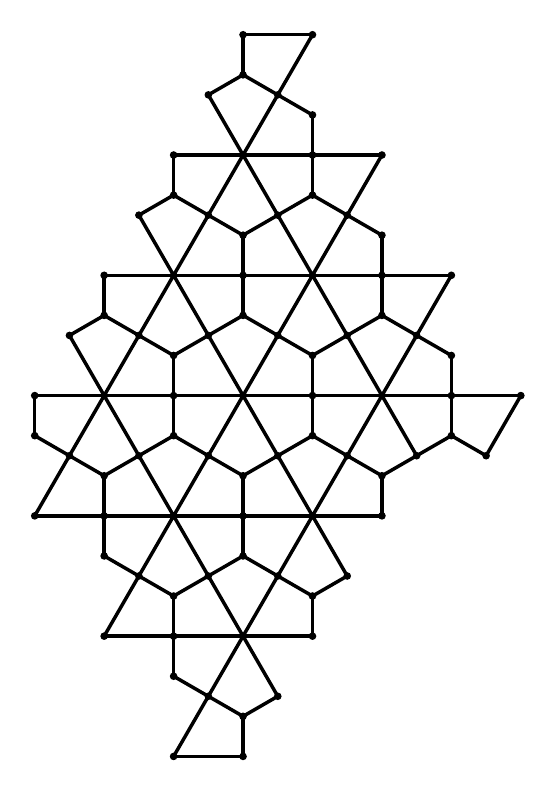}
\end{center}
\end{figure}

This paper follows the approach of~\cite{eklp}, but on a different lattice: namely, the hexagon lattice superimposed upon its dual triangular lattice (see Figure~\ref{fig:the dp3 lattice}).  We
name this lattice the \emph{$dP_3$ lattice}, since it is associated to the del Pezzo
3 surface (the nature of this association is explained in
Section~\ref{sec:height function}).  This association was first made clear in the physics literature~\cite{franco} in a quite general setting; it is from this paper that we take the terminology $dP_3$.  Indeed,~\cite{franco} introduces several quivers which correspond to the $dP_3$ surface; ours is their Model I.

We describe a type of domino shuffling
which works on the $dP_3$ lattice.  We also define an analogue of the Aztec diamond
of order $n \in \frac{1}{2}\bbZ$ for this lattice (see Figure~\ref{fig:some
diamonds}), which is well behaved under domino shuffling, and prove that the
number of perfect matchings on these graphs is a power of 2:  

\begin{theorem}
\label{thm:main}
Let $m \in \frac{1}{2}\bbZ$ and let $n = \lfloor m \rfloor$. Then
the number of perfect matchings on a diamond of order $m \in \frac{1}{2}\bbZ$ is 
\begin{align*}
&\left\{
\begin{array}{cl} 
 2^{n(n+1)} & \text{if }m \in \bbN  \\
 2^{(n+1)^2} & \text{if }m + \frac{1}{2} \in \bbN.
\end{array}
\right. 
\end{align*}
\end{theorem}

Moreover, we can give a generating function for these matchings; we defer the statement of that theorem until Section~\ref{sec:generating_functions}, due to the number of definitions required to state the result.  

We constructed our domino shuffle by reverse engineering a particular
application of urban renewal,  and replacing the averaging procedure (replacing it
instead with a careful accounting of the local structure of the perfect
matching).

Propp later
refined his domino shuffle in~\cite{propp-2003}  by reworking domino shuffling so that it can be
applied locally, to a square face of a weighted planar bipartite graph. The resulting technique,
\emph{urban renewal}, involves
altering the graph at a square face S, as well as averaging the weights of the
edges around S in a certain way. Ciucu gave a further refinement to this
formalism, reproving the results in~\cite{ciucu}  as well as many others.
Unbeknownst to us when we began this work, Ciucu had already proved half of
our Theorem~\ref{thm:main} in~\cite{ciucu} using a powerful reduction theorem. Moreover, there
is a type of domino shuffle that one can do on any finite planar bipartite
graph G: simply embed G inside a suitably large Aztec diamond and use the generalized
domino shuffle~\cite{propp-2003} on the square lattice.  It is therefore important to note that
our main contribution is the construction of the domino shuffle on the $dP_3$
lattice itself, rather than its applications to counting Aztec diamonds. Our
shuffle is indeed different from the ones used in the aforementioned papers.   

There are at least two good reasons to look for simple domino
shuffles on particular planar bipartite graphs rather than appealing to general
theory. The first of these came to light in~\cite{szendroi, young-2009}, and
involves a curious interplay between the combinatorics of perfect matchings and
algebraic geometry. By studying a particular class of perfect matchings on
(infinite) periodic planar bipartite graphs, one can compute a
\emph{Donaldson-Thomas (DT) partition function} of the toric surface
associated to the graph by the construction in~\cite{franco}.  The DT partition function is a virtual count
of curves in the space (really a type of Euler characteristic on the moduli
space of subschemes of the surface). There is an emerging picture in which
domino shuffling corresponds to a geometric procedure called \emph{wall crossing}
\cite{bryan-cadman-young, kontsevich-soibelman, nagao-nakajima, bryan-young}, and which, for example, links the DT theory of a singular surface to that of its
crepant resolution. 
The generating functions that we compute in Section~\ref{sec:generating_functions} are a prelude to our next project, wherein we will compute the Donaldson-Thomas partition function for this lattice; this turns out to be a generating function for a different type of dimer configuration on this lattice, but which uses essentially the same weighting.  We do not address Donaldson-Thomas theory at all in this article, but see~\cite{szendroi, young-2009} for an account of the corresponding problem on the square lattice.

The second modern setting for domino shuffling is in the
work of ~\cite{nordenstam, borodin-gorin} and others, which describe square-lattice domino shuffling
as a type of random point process. This process turns out to be partially
determinantal, which makes it one of a select few three-dimensional statistical
physical models which can be treated completely mathematically, without applying to 
heuristic physical reasoning.  One current frontier of exactly solvable statistical
mechanics is precisely this jump from two to three
dimensions, and it is quite important to understand fully the phenomenology of
any tractable models in this area. It seems very likely that our domino shuffle
fits into the framework of~\cite{borodin-gorin}, but we have yet to exploit this connection.  

In a very recent string theory paper~\cite{franco-eager-mauricio}, Franco, Eager and Romo have reproduced (up to a change of variables) and substantially generalized our Theorem~\ref{thm:main}, using different methods.  We gratefully acknowledge their help in spotting a typo in Theorem~\ref{thm:main} in an earlier version of this preprint.

\section{The Lattice}

Before taking the time to describe these subgraphs, we must better understand the $dP_3$ lattice itself (see Figure~\ref{fig:the dp3 lattice}).
 It is periodic, bipartite, and planar, where each face is bounded by a cycle of size 4. The edges incident to centers of hexagons have length 1, and other edges have length $\sqrt3/3$.  Color the vertices of degree four black, and the others white.

\begin{definition}
We define a \emph{square} to be a face bounded by a cycle of size 4. A \emph{long edge} of a square is one of length $1$; a \emph{short edge} has length $\sqrt{3} / 3$.  Every square has a unique vertex of degree 3. The \emph{tail} of a given square is the edge which is incident to this unique vertex and not on the boundary of the square. It is incident to another vertex which is not the center of a hexagon.   See Figure~\ref{fig:short edge, long edge, tail, square and kite}.
\end{definition}

 \begin{figure}
\caption{Short edge, long edge, tail, square and kite; orientations.
\label{fig:short edge, long edge, tail, square and kite}
\label{fig:pairs of orientations}}
\begin{center}
\includegraphics[height=1in]{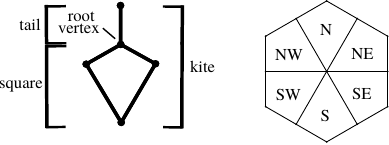}
\end{center}
\end{figure}

 \begin{figure}
\caption{A strip
\label{fig:a strip}}
 \begin{center}
\includegraphics[height= 1in]{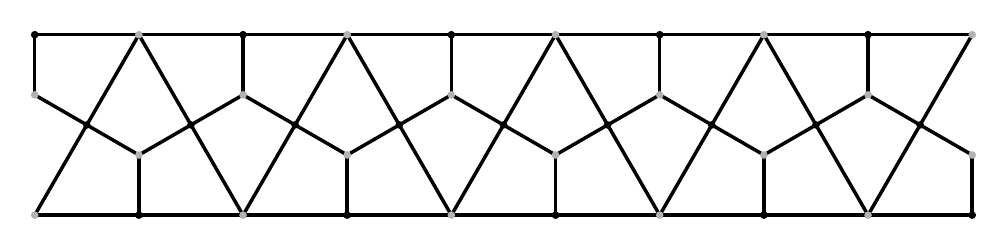}
\end{center}
\end{figure}

We orient our lattice such that for each hexagon, a white vertex is directly north of the center of the hexagon. The square which is bounded by both the center and this white vertex has orientation N, and the square directly south has orientation S.  We shall often have cause to pair up orientations which are in opposite directions on the compass:  N/S, NE/SW, NW/SE, and we shall speak of these as \emph{pairs of orientations}.
 
\begin{definition}
A \emph{kite} is a square together its tail, and a \emph{strip} of kites is a set of adjacent kites with the same pair of orientations, which meet at a degree-four vertex of their squares (not their tail).  Given a kite $K$, let the unique vertex of degree 3 that is on the boundary of the square and incident to the tail be $K$'s \emph{root vertex}.  See Figure~\ref{fig:short edge, long edge, tail, square and kite}.
\end{definition}

Note that if a kite has its square oriented N/S, and the tail is north of the square, we have a N kite, while we have a S kite if the tail is south of the square; the same holds for NE, SW, NW, SE kites. 

\begin{definition}
A \emph{matching} is a set of vertex-disjoint edges of a graph $G$. It is \emph{perfect} when it covers every vertex of $G$.
\end{definition}

Two graph transformations which are often used in perfect matching enumeration are urban renewal and double-edge contraction (see Figure ~\ref{fig:urban renewal and double edge contraction}).  Double-edge contraction is the simpler of the two techniques, in which a vertex of degree two is merged with its two neighbors; it is easy to check that this transformation induces a bijection on perfect matchings.  

Urban renewal is slightly more complicated; it replaces a square $S$ in the graph with a smaller square $S'$, together with four new edges connecting the vertices of $S$ to the vertices of $S'$ (see Figure~\ref{fig:urban renewal and double edge contraction}).  When applied to an edge-weighted graph, it is possible to define new weights in such a way that urban renewal leaves the perfect matching generating function invariant up to a constant~\cite{propp-2003}.  The reader should note that in our paper, we do \emph{not} do this change of edge weights, and indeed we are often not working with weighted graphs at all; for us, ``urban renewal'' simply means the graph transformation shown in Figure~\ref{fig:urban renewal and double edge contraction}).

\begin{figure}
\caption{Urban renewal and double edge contraction
\label{fig:urban renewal and double edge contraction}}
\includegraphics[height=1.5in]{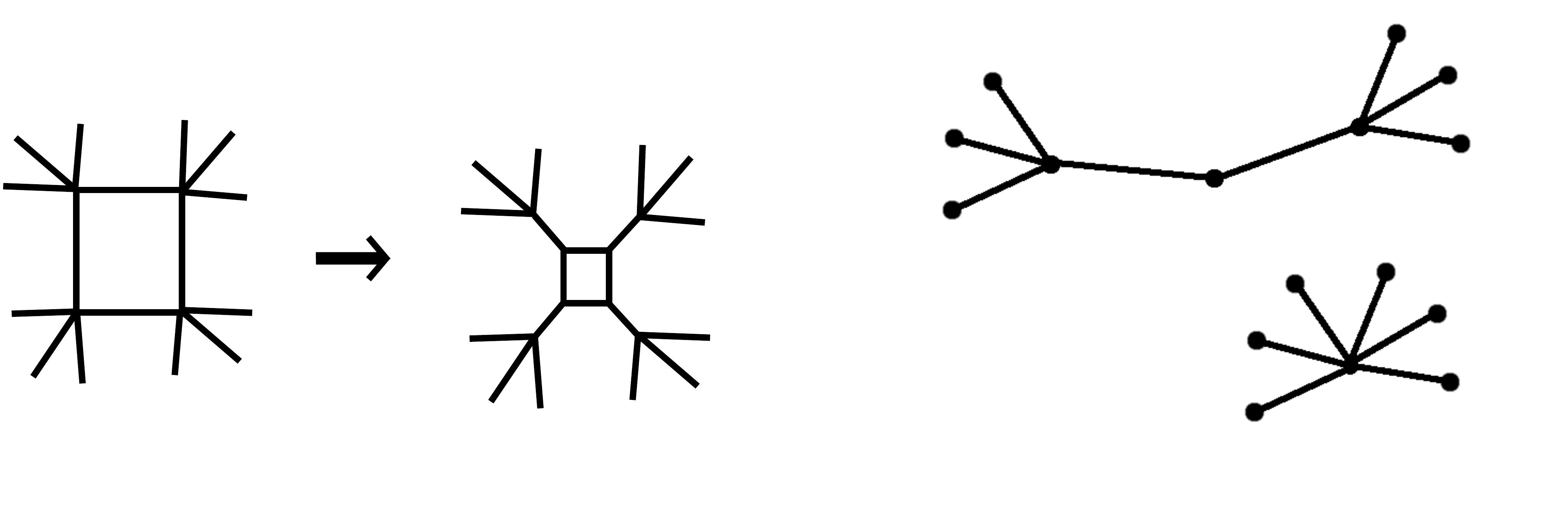}
\end{figure}

\section{The Integer-order Diamonds}

We now define a family of subgraphs called \emph{diamonds}.  This family is indexed by the\emph{order} of the diamonds, which will be an element of $\frac{1}{2}\bbZ_{\geq 0}$.  
Our shuffling algorithm will shuffle the edges on a perfect matching on an order $n$ diamond, and outputs a perfect matching on an order-$(n+1)$ diamond. 

We begin by discussing the integer-order diamonds, which were first discussed by Propp~\cite{propp-1999} and further studied by Ciucu~\cite{ciucu}, under the name \emph{Aztec dragons}.

\begin{definition}
A \emph{diamond of order 1} is denoted $D_1$ (see Figure~\ref{fig:order 1 diamond}); it is a connected subgraph of $dP_3$ which consists of 3 squares. One of them is N, while the other two are oriented NW and SE and are adjacent on the NE/SW strip they lie on. 
 Each diamond contains a unique degree-four vertex which is on the boundary of all of its squares. We call this vertex the \emph{center vertex} of a $D_1$.
\end{definition}

 \begin{figure}
\caption{The diamond of order 1, $D_1$
\label{fig:order 1 diamond}}
 \begin{center}
\includegraphics[height=1in]{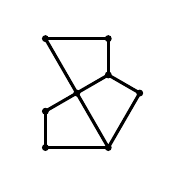}
\end{center}
\end{figure}

Note that a N-S strip is tiled by copies of $D_1$, half of which have been rotated 180 degrees.  We refer to these as \emph{North} and \emph{South $D_1$s}, according to the direction of their $N-S$ kite.

Now, we develop a coordinate system such that we can define a diamond of order $n$, $n \in \mathbb{N}$.  Recall that the $dP_3$ lattice decomposes into strips, each of which can be tiled by copies of $D_1$.

\begin{definition}
\label{defn:dn_integer_coords}
First, fix in the lattice a $D_1$ and call it $T(0,0)$. If $n$ is even, let $T(0,0)$ be oriented South, and if $n$ is odd, let $T(0,0)$ be oriented North.  Call $T(1,0)$ the $D_1$ directly east of $T(0,0)$, and $T(-1,0)$ the $D_1$ directly west of $T(0,0)$. Similarly, the $D_1$ directly north of $T(0,0)$ is $T(1,0)$ and the one directly South of $T(0,0)$ is $T(-1,0)$. If $i,j \in \mathbb{Z}$, then $T(i,j)$ must be directly north of $T(i,j-1)$, directly south of $T(i,j+1)$, directly east of $T(i-1,j)$, and directly west of $T(i+1,j)$.

For $i \in \mathbb{N}$, let $S_i$ be the strip containing $T(0,i)$.

Finally, let 
\[D_n = \displaystyle\bigcup_{\substack{-n < i,j < n \\ |i| +|j| < n}} T(i,j),\] where the union denotes the union of vertex and edge sets of subgraphs of $dP_3$.
\end{definition}

\begin{figure}[h]
\caption{Coordinates on an integer- and half-integer-order diamonds.  In both, $T(2,1)$ is shaded.  Not all graph edges are shown.
\label{fig:coordinates on a whole diamond}
\label{fig:constructing a half diamond}}
\begin{center}
\includegraphics[height=2.80in]{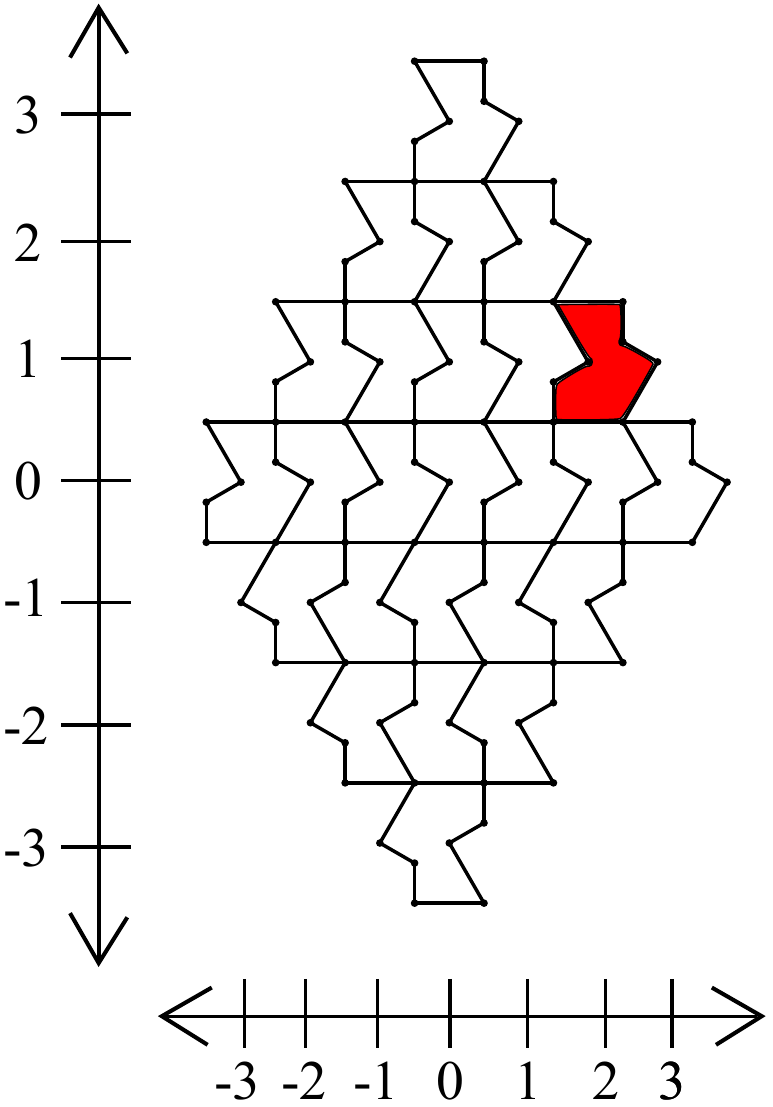}
\rule{0.5in}{0in}
\includegraphics[height=2.90in]{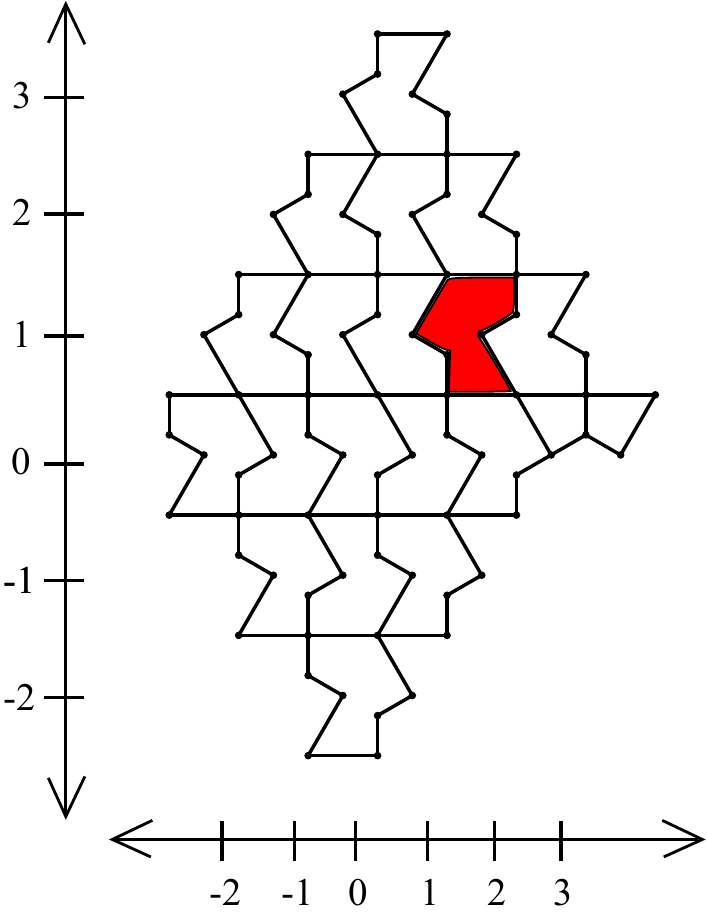}
\end{center}
\end{figure}

\section{The Half-integer-order Diamonds}
We construct a second, new family of finite subgraphs of $dP_3$. These are \emph{half-integer-order diamonds}, $D_{n + \frac{1}{2}}$ where $n \in \mathbb{N}$.  

Again we describe them with a coordinate system. 

\begin{definition}
Let $T(0,0)$ be oriented North when $n$ is even, and let it be oriented South when $n$ is odd.  Let $D_1'$ denote a $D_1$ which has been reflected in the $y$ axis.  Decompose the N-S strips into copies of $D_1'$; let $T(i,j)$ denote the $D_1'$ which is $i$ steps east and $j$ steps north of $T(0,0)$ as in Definition~\ref{defn:dn_integer_coords}. 

Define $L_n$ to be the SE-oriented square of $T(n,0)$, and $R_n$ to be the SW-oriented square of $T(n+1,0)$. Let $D_{n + \frac{1}{2}}$ be the following graph (see figure ~\ref{fig:constructing a half diamond}):

\begin{align*}
D_{n+\frac{1}{2}} =  \left( \displaystyle\bigcup_{\substack{-n < i,j < n \\ |i| +|j| < n}} T(i,j) \right) \cup \left( \displaystyle\bigcup_{\substack{0 < j \le n \\ 0 < i < n + 1 + j}} T(i,j) \right) \cup R_n \cup L_n.
\end{align*}

\end{definition}

Note that these graphs no longer decompose completely into copies of $D_1'$:  there are two extra squares, $L_n$ and $R_n$.  Finally, we specifically describe $D_{\frac{1}{2}}$.

\begin{definition}
\emph{$D_{\frac{1}{2}}$} consists of a single NE square.
\end{definition}

\begin{definition}
Let \emph{$M_n$} be the number of edges in a perfect matching on a diamond of order $n$, such that $n \in \bbN$ or $ (n + \frac{1}{2}) \in \bbN$. 
\end{definition}

\begin{lemma}
Let $n = \lfloor m \rfloor$.  Then 
\begin{center}
$ M_m = \left\{
\begin{array}{cl} 
{n(3n+1)} &  m \in \bbN  \\
{3n^2 + 4n + 2} &  (m + \frac{1}{2}) \in \bbN.
\end{array}
\right. $
\end{center}

\end{lemma}

\begin{proof} The number of edges in a perfect matching is half of the number of vertices in the graph, so this is a straightforward computation using Euler's formula for planar graphs and the definitions of the graphs $D_n$.
\end{proof}

\section{The Rules of the Shuffling Algorithm}

Throughout this section, we let $n = \lfloor m \rfloor$. 

We define a procedure called \emph{domino shuffling}, whose input is a perfect matching $M$ on a $D_m$, and whose output  output is a perfect matching on $D_{m + \frac{1}{2}}$.
We were inspired to work out this algorithm by the following observation:  

\begin{figure}
 \caption{Applying urban renewal and double-edge contraction to a kite
\label{fig: applying urban renewal and double edge contraction to a kite}}
\begin{center}
\includegraphics[height=2in]{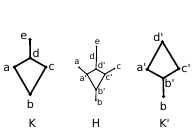}
\end{center}
\end{figure}

\begin{lemma}
\label{lem:lattice shift}
Urban renewal applied to all kites of a given pair of orientations, followed by contracting double edges wherever possible, simply translates the $dP_3$ lattice. 
\end{lemma}

\begin{proof}
Without loss of generality, let us perform urban renewal on North and South kites; the same argument applies to NE/SW and NW/SE kites.  

 First of all, we choose a N/S kite and call it $K$; its square has vertices $a,b,c,d$ and its tail has edge $de$. Then, apply urban renewal to all of the N/S kites (see Figure ~\ref{fig: applying urban renewal and double edge contraction to a kite}). For each kite we have a new square $a',b',c',d'$ and 4 new edges $aa', bb', cc', dd'$. 

We consider the new edges associated with $K$. We observe that one of the new edges, $dd'$, is adjacent to the root vertex of $K$. Two other new edges, $aa'$ and $cc'$, are adjacent to the new edges of the N/S kites which lie on the same N/S strip as $K$.  We now apply double-edge contraction to the vertices $a,c,d$. We do this to every kite we applied urban renewal to. Hence, for each kite we have deleted 1 old and 3 new edges, and added 1 new edge. This new edge happens to be $bb'$, which is adjacent to a vertex opposite the original root vertex. Therefore, the new figure is in fact a kite with tail $bb'$, which is oriented in the opposite direction of the original kite. We have flipped the kite. Moreover, this flip has occurred on every $N-S$ kite, hence we have shifted the lattice.
\end{proof}
 
 \begin{figure}
\caption{Applying urban renewal and double-edge contraction to N$\backslash$S kites in a strip.
\label{fig:applying urban renewal to dP3 lattice}}
 \begin{center}
\includegraphics[height=0.88in]{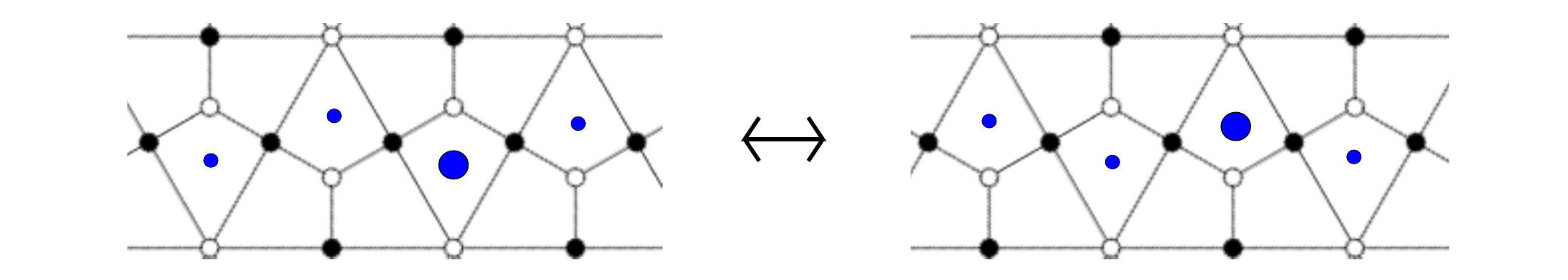}
\end{center}
\end{figure}

Now, we proceed to define the algorithm as applied to a perfect matching $M$ on $D_m$.  Let $n = \lfloor m \rfloor$.

 \begin{figure}
\caption{Steps 1 through 3 of shuffling} 
 \begin{center}
\includegraphics[height=2in]{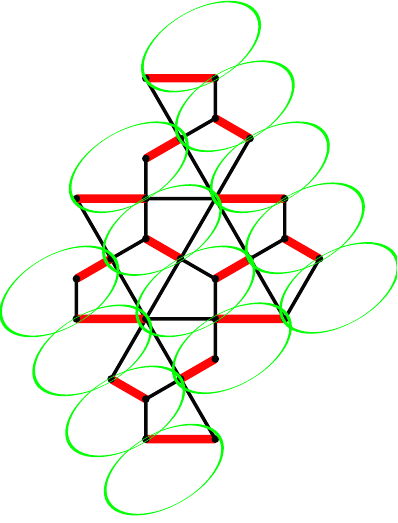}
\includegraphics[height=2in]{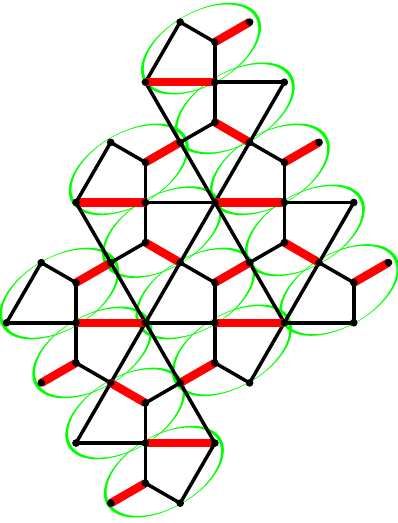}
\includegraphics{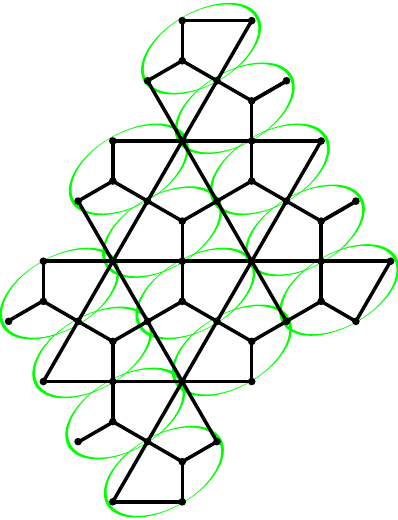}
\end{center}
\end{figure}

\begin{algorithm}[The domino shuffle]  Begin with a perfect matching on $D_m$.
\label{alg:domino shuffle}

\begin{enumerate}
\item If $m$ is an integer, draw an oval around every $NE-SW$ kite in $D_m$, including partial kites.  Otherwise, draw an oval around every $NW-SE$ kite, including partial kites.  
\item Fill in all of the partial kites around the boundary, by adding new edges to the diamond. None of these new edges should be added to the matching, with the following exception: the root vertex of each circled kite is matched. If not, add the tail of that kite to $M$. We will add $2n +1$ tails, by lemma 6 below.
\item Flip each kite in each oval so that it has an opposite orientation.
\item Draw the new matching $M'$ using the rules in Algorithm~\ref{alg:edge rules} (see below) and shown in Figure~\ref{fig:rules of domino shuffling on a kite}.
\item Drop the unmatched vertices on the boundary ovals (the tail on every other kite is never part of the matching.)
\end{enumerate}
\end{algorithm}

Observe that the edges not on $K$ are not altered by our algorithm: they remain as they are (either matched or unmatched).

\begin{figure}
 \caption{Steps 4 and 5 of the domino shuffle
\label{fig:steps 5 and 6 of shuffling}
\label{fig:rules of domino shuffling on a kite}}
\begin{center}
\includegraphics[width=2.8in]{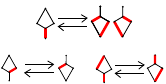}
\rule{0.5in}{0in}
\includegraphics[width=1.3in]{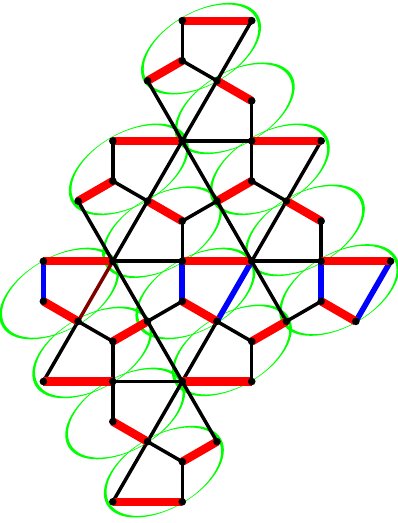}
\end{center}
\end{figure}

As stated earlier, the shuffling is inspired by applying urban renewal and double contraction to the squares of the circled kites (but not changing any of the weights on these edges -- indeed, all of this will work on an unweighted graph), and keeping track of what happens to the associated perfect matching. 
The rules are shown graphically in Figure~\ref{fig:rules of domino shuffling on a kite}, which is perhaps the most useful format.  However, to ensure clarity, we will also describe the rules more formally.

 We call our kite $K$ and adopt the notation of Lemma~\ref{lem:lattice shift}, illustrated in Figure~\ref{fig: applying urban renewal and double edge contraction to a kite}:  the vertices of $K$ are $a,b,c,d,e$ and after urban renewal, the square of $K$ is replaced with a new graph $H$ with vertices $a,a',b,b',c,c',d,d'$.

\begin{algorithm}[Rules for edges] 
\label{alg:edge rules}

Let $K$ be a kite with edges as described above.

\begin{enumerate}
\item[Case 1:] No edges in $K$ are in $M$.

Then, our root vertex is not matched as it is not adjacent to any edges outside, so this case does not occur.

\item[Case 2:] 1 edge in $K$ is in $M$.
\begin{itemize}
\item[a.] Suppose that the tail $ed$ is in $M$. Take either $a'b'$ and $c'd'$, or or $a'd'$ and $c'b'$ to be in the new matching. This is called a \emph{creation}.
\item[b.] Suppose that the tail is not in $M$. Then $M$ contains a short edge, either $ad$ or $dc$.  Take $b'c'$ (respectively, $a'b'$) to be in the matching. We call this the \emph{short edge} case.
\end{itemize}
\item[Case 3:] 2 edges in $K$ are in $M$.
\begin{itemize}
\item[a.] Suppose that the tail of $K$ is in $M$. Then, either $bc$ or $ab$ is in $M$.  Take $a'd'$ (respectively, $c'd'$) to be in the new matching, as well as the new tail $bb'$. We call this case the \emph{long edge} case.
\item[b.] Suppose that the tail of $K$ is not in $M$. Then take only the new tail $bb'$ to be in the new matching.
We call this case an \emph{annihilation}. 
\end{itemize}
\end{enumerate}
\end{algorithm}

Let us analyze this lemma, to see that it works as advertised.
\begin{lemma}  
Shuffling a perfect matching on all of $dP_3$, choosing one of the two outcomes of each creation arbitrarily, produces another perfect matching on the translation of $dP_3$. 
\end{lemma}

\begin{proof}
We examine one oval, and proceed by cases as in Algorithm~\ref{alg:edge rules}.  Our strategy is to perform the operations of urban renewal and double edge contraction, each of which comes with an induced map on perfect matchings.  In each case, we verify that we have duplicated the shuffling rules of Algorithm~\ref{alg:edge rules}.

\begin{enumerate}
\item[Case 1:] No edges in $K$ are in $M$.

As remarked above, this is not possible.

\item[Case 2:] 1 edge in $K$ is in $M$.
\begin{itemize}
\item[a.] Suppose that the tail $ed$ is in $M$. Then, $a,b,$ and $c$ are matched outside of $K$. When we apply urban renewal, we must match the vertices on $S'$. Hence, we can take either $a'b'$ and $c'd'$ or $a'd'$ and $c'b'$. There are two choices. When we apply double-edge contraction we see that $K'$ will have either $a'b'$ and $c'd'$ or $a'd'$ and $c'b'$ in its matching. 
\item[b.] Suppose that the tail is not in $M$. In order to ensure that our root vertex is matched, we only have a case where our matching contains a short edge, either $ad$ or $dc$. Without loss of generality, suppose that this edge is $ad$. Then, $b,c,$ and $e$ are matched outside of $K$. We only need to match $S'$ and $a$ and $d$ on $H$. Hence, we let $aa'$ and $dd'$ be in the matching as well as $b'c'$. Then, we have a perfect matching. Finally, we apply double-edge contraction and obtain $K'$ such that only $b'c'$ is in the matching. 
\end{itemize}
\item[Case 3:] 2 edges in $K$ are in $M$.
\begin{itemize}
\item[a.] Suppose that the tail of $K$ is in $M$. Then, we must have a long edge in $M$. Without loss of generality, let this edge be $bc$. Then, we have that $a$ is matched outside of $K$. We must now match $S'$ and $b$ and $c$ in $H$. Thus, we take the edges $a'd'$, $bb'$, and $cc'$. Applying double-edge contraction, we get that the edges in $K'$ which are in our new matching are $a'd'$ and $b'b$. 
\item[b.] Suppose that the tail of $K$ is not in $M$. Then, the 2 edges in $M$ must be a short edge and a long edge on $S$. Without loss of generality, we let these two edges be $ad$ and $bc$; the other possibility is choosing $dc$ and $ab$. Now, we have that $e$ is matched outside of $K$, so we include the edges $aa'$, $bb'$, $cc'$, and $dd'$. Then, when we apply double-edge contraction, we will only have $bb'$ in our new matching. We see that this occurs no matter which pair of edges on $S$ we select. 
\end{itemize}
\end{enumerate}

\end{proof}

\begin{lemma}
When applying step 3 of our shuffling algorithm, we add $2n + 1$ tails to ensure that every root vertex is matched.
\end{lemma}

\begin{proof} Consider first the case $n \in \bbZ$.
For each strip, we see that the two $D_1$s on the boundary are either both oriented N, or both oriented S. If they are both N diamonds, we circle two kites on the eastern boundary, and one on the western boundary. If they are both S diamonds, we circle two kites on the western boundary, and one on the eastern boundary. Two of these kites have their tails in $D_m$, while the third one only has one long edge in $D_m$. Hence, its root vertex will never be matched. So, we will have to add as many tails as the number of strips. 

Finally, note that there are two more kites on the boundary. One is north of $T(0,n-1))$ the other is south of $T(0,-n+1)$. Each of these kites only have one long edge in $D_m$, hence their root vertex also must be matched.

Thus, we add exactly $2n-1 + 2 = 2n+1$ vertices.

The same argument holds for half-integer-order diamonds, if we replace $D_1$ by $D_1'$ throughout.
\end{proof}

\begin{lemma}
Given a perfect matching $M$ on $D_m$, applying shuffling to $M$ gives us a matching $M'$ which covers all of the vertices on a diamond of order $m + \frac{1}{2}$.
\end{lemma}

\begin{proof}
First of all, by Lemma~\ref{lem:lattice shift}, domino shuffling applied to the entire lattice behaves as a simple shift. Hence we simply need to show that the number of strips on the new subgraph is correct, that the correct types of squares are present on the boundary, and that step (5) is deterministic and removes only tail vertices from the boundary.
We will do in detail the case that $m=n$ are integers; the half-integer case is
similar.  We encourage the reader to refer to figures~\ref{fig:some diamonds} and~\ref{defn:dn_integer_coords} to verify these claims.

First, note that an N square transforms into a SE square and \textit{vice versa}; a S square becomes a NW square and \textit{vice versa}; and a SW square becomes a NE square and \textit{vice versa}. 

  When we flip the kite directly north of $T(0,n-1)$ we have added a
square that is in a strip which was previously not included in our subgraph.
This new square is a SW, since the original is NE.  Similarly, when we flip the
kite which is directly south of $T(0,-(n-1))$, we include a square which was
not on any strips that were on our original graph. The new square will be NE,
since the original was SW.  We see that we have increased the number of
strips by 1.

Now, we check that our new boundary is correct. 
We start with the northernmost kite that is added in Step 2: it is a NE kite. Hence, the square we obtain after shuffling is a SW square. The northernmost square on the original subgraph was a SE square, and this one will become a N square. There is also an adjacent SW square in the original graph which becomes a NE square. These three new squares share a common vertex which becomes the center vertex of an eastern N $D_1$; these form the correct shape for the top corner of a half-integer-order diamond.  Moreover, one can apply this to the $D_1$'s on the strips which lie on or to the north of the center strip $S_0$. 

We can apply a similar argument for the strips which lie south of $S_0$. We start with the southernmost square which is added in Step 2: it is a SW kite. Hence, the square we obtain after shuffling is a NE square. The southernmost square on the original subgraph was a NW square, and this one will become a S square. Similarly, we added a NE kite in Step 2; such that its square will be a SW kite after we apply shuffling.  These three new squares are such that they share a common vertex which becomes the center vertex of a western S $D_1$. Hence, we have the exact shape which we were looking for. One can apply this to the $D_1$s on the strips which lie south of $S_0$. 

Next, we check that step (5), in which unmatched boundary vertices are removed, is completely deterministic.  Indeed, all tail vertices on the boundary are always unmatched after shuffling.  This follows from an inspection of the shuffling rules. Let $K$ be a kite.  Observe that tail is unmatched within $K$ after shuffling if the vertex $v$ incident to the two long edges of $K$ is matched within $v$ \emph{before} shuffling, which is always the case when $v$ is on the boundary of $D_n$.
\end{proof}

\section{Proof of Theorem 1}

Our proof is essentially an analysis of the domino shuffling algorithm.  In short, the algorithm behaves much like a $1$-to-$2^{|n|+1}$ correspondence for enumerative purposes.

\begin{proof}
Let $n = \lfloor m \rfloor$.
When we apply the shuffling algorithm to $M \in PM(D_m)$, we add $2n + 1$ tails, but we know that the associated matching on $D_{m + \frac{1}{2}}$ has $3n + 2$ more edges than those in $M$. We can account for $2n + 1$ of these new edges by the $2n + 1$ tails added when applying our algorithm. Hence, the difference between the number of creations and number of annihilations is $n+1$.

Suppose that $M$ has $k$ annihilations. There is a class of exactly $2^k$ tilings which have $k$ annihilations on these kites, and are otherwise identical to $M$ .  Moreover, $M$ must experience $k + n + 1$ creations after shuffling.
On each kite which had a creation, there are 2 possible outcomes. Hence, the $2^k$ matchings with $k$ annihilations map to exactly $2^{k+n+1}$ matchings on $PM(D_{m + \frac{1}{2}})$. As such,
\begin{align*}
|PM(D_{n+1})| &= 2^{n+1} |PM(D_{n+\frac{1}{2}})|; \\
|PM(D_{n+\frac{1}{2}})| &= 2^{n+1} |PM(D_{n})|.
\end{align*}

The count of perfect matchings now follows by a standard induction argument, the base cases being $PM(D_0) = 1$, because the empty graph has only the empty matching, and $PM(D_{\frac{1}{2}}) = 2$, because the order $\frac{1}{2}$ diamond is a $SW$ square and has two matchings (see Figure~\ref{fig:perfect matchings on order half diamond}).
\end{proof}
 
\begin{figure}
\caption{Case $m = \frac{1}{2}$
\label{fig:perfect matchings on order half diamond}}
 \begin{center}
\includegraphics[height=1in]{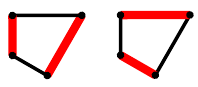}
\end{center}
\end{figure}

\section{Height function}
\label{sec:height function}

Our next aim is to give generating functions for the perfect matchings on $D_n$. To describe the statistics encoded in our generating functions, we will need to introduce the notion of \emph{height function}.  This is a rather general concept associated to any perfect matching on a planar bipartite graph (see, for example,~\cite{kenyon}), but it is necessary to work out the precise nature of the height function on $dP_3$ to proceed further.

\begin{figure}
\caption{A fundamental domain of the dual quiver of the $dP_3$ lattice
\label{fig:dualquiver}}
\begin{center}
\includegraphics[width=1.5in]{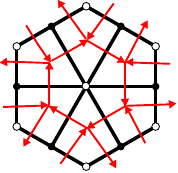}
\end{center}
\end{figure}

If a planar graph $G$ is bipartite, its dual graph $Q$ inherits a natural orientation: orient each edge of $Q$ so that a black vertex of $G$ is on the left (say).  We use the letter $Q$ because the dual graph should be regarded as a \emph{quiver}~\cite{szendroi, young-2009}, but for combinatorists a quiver is precisely the same structure as a directed graph. If we do this construction for the $dP_3$ graph, we obtain an orientation on the (6,4,3) lattice with arrows pointing counterclockwise on hexagons and triangles, and clockwise on squares (see Figure ~\ref{fig:dualquiver}).

We will assign an integer to each arrow $e$ of $G'$ as follows:
\begin{equation}
H(e)= \begin{cases}
1 & \text{if $e$ crosses a long edge} \\
2 & \text{if $e$ crosses a short edge}
\end{cases}
\end{equation}
This integer is called the \emph{height change} of $e$.  Moreover, given a perfect matching $M$ on $G$, we will define a function $h:V(G') \rightarrow \bbZ$ as follows:  fix $a$ to be one of the vertices in $G'$ and set $H(a) = 0$.  Now, if $e:b \rightarrow c$ is an edge of $g'$ and $h(b)$ has been defined, set
\begin{equation}
h(c) = 
\begin{cases}
h(a) + H(e) - 6 &\text{if }e \in M \\
h(a) + H(a) &\text{if }e \not \in M.
\end{cases}
\end{equation}
It is easy to check that $h$ is well-defined up to the initial choice of height $h(a)$. This is essentially the same height function defined in~\cite{kenyon}.

\begin{definition}
Let $f$ be a face of $G$, bounded by edges $E = \{e_1, \cdots, e_{2k}\}$.  Let $M$ be a perfect matching on $G$ such that $M \cap E = \{e_1, \ldots, e_{2k-1}\}$.  The \emph{plaquette flip of $M$ around $f$} is the perfect matching $M'$ which agrees with $M$ except around $f$, where $M' \cap f = \{e_2, \ldots, e_{2k} \}$.
\end{definition}

It is easy to check that performing a plaquette flip at a kite $S$ changes the height at $S$ by $\pm 6$, leaving the height function otherwise unchanged.  Observe also that one can transform any perfect matching to any other by a sequence of plauqette flips.  This is because if $h_1, h_2$ are height functions for two perfect matchings, then so is $min(h_1, h_2)$; one constructs its associated perfect matching by superimposing $h_1$ and $h_2$, producing a collection of disjoint doubled edges and loops of even length.  

There are two ways to remove every second edge from a loop, leaving a perfect matching on the vertices of the loop; for each loop, remove the set of edges which minimizes the resulting height function inside the loop.
Doing this, we obtain the matching associated to $min(h_1, h_2)$.

In particular, the function
\[
\min_{M \in PM_n} h(M)
\]
is a height function.  Observe that this matching admits only plaquette moves which increase $h$.

\begin{figure}
\caption{Tiles for encoding the height function on the $dP_3$ lattice.
\label{fig:tiles}}
\includegraphics[width=3in]{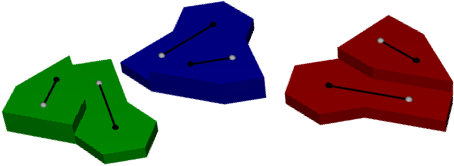}
\end{figure}

\begin{figure}
\caption{A perfect matching, viewed as a stack of bricks in a tray.
\label{fig:traytiles}}
\includegraphics[width=1.6in]{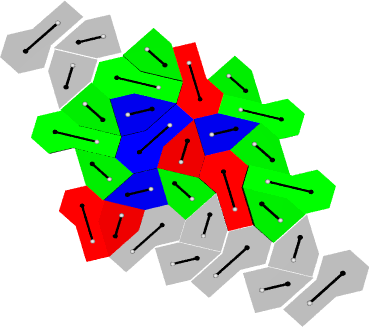}
\includegraphics[width=1.6in]{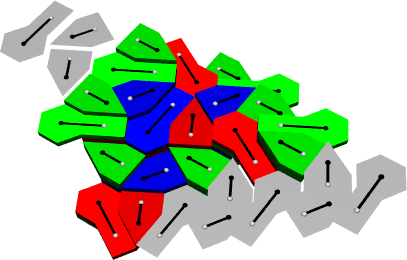}
\includegraphics[width=1.6in]{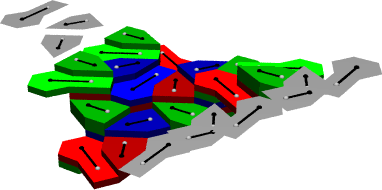}

\end{figure}

In light of these observations, we may visualize a perfect matching on $D_n$ as a pile of special bricks, contained in a V-shaped tray (just as is done in~\cite{eklp}).  Plaquette flipping corresponds to adding a brick, and the tray encodes the shape of the minimal height function(see Figure ~\ref{fig:traytiles}. 
See Figure~\ref{fig:tiles} for a picture of the tiles.  They are constructed so as to meet the following criteria:

\begin{enumerate}
\item Each tile covers exactly the four faces in $Q$ which are dual to the vertices composing a square in $G$;
\item If one tile is placed partially atop another, covering two adjacent kites $a$, $b$ of $G$ joined by a dual arrow $e$ in $Q$, then the vertical distance between the two tiles is the height change $h(e)$; 
\item The tiles, viewed from above, have two dimers stencilled on top: the same two dimers which are present \emph{after} a plaquette flip on the corresponding kite has been performed, so as to \emph{increase} the height function there.
\end{enumerate}
Any tiles meeting these criteria will work to reproduce perfect matchings on $D_n$. The tiles we have chosen are the simplest possible ones: all of the tiles are identical, all of the faces are flat, and the underside of each tile is planar.  We have colored each tile according to its orientation.

This construction appears to be highly artificial, so we should say a few words to explain why it is not.  Our notion of height function comes from noncommutative Donaldson-Thomas theory~\cite{szendroi}, in which we study the representation theory of a certain quotient of the path algebra $\mathbb{C}Q$, the algebra of directed paths in $Q$ (multiplication being given by concatenation when the paths share an endpoint, and the zero map when they do not).  This theory was worked out first in the physics literature in~\cite{franco}.  In particular, let 
$\mathcal{P}_{CCW}$ and $\mathcal{P}_{CCW}$ denote the set of counterclockwise (respectively, clockwise) directed paths around single faces in $Q$, and let 
$W \in \bbC_Q / [\bbC Q, \bbC Q]$ be the \emph{superpotential}
\[
\sum_{p \in \mathcal{P}_{CCW}} \prod_{a \in p} a - 
\sum_{p \in \mathcal{P}_{CW}}  \prod_{a \in p} a.
\]
Let $I_w$ be the ideal generated by partial derivatives $\partial W / \partial a$ for each arrow $a \in Q$.  The algebra that we study in noncommutative Donaldson-Thomas theory is $\bbC Q / I_W$.  In this algebra, any two closed loops around a face with the same endpoint are identified; moreover, if $p$ is a path from $a$ to $b$, and $\ell_a$ and $\ell_b$ are loops enclosing one face rooted at $a$ and $b$ respectively, one can show that
\begin{equation}
\label{eqn:moving a loop along a path}
[\ell_b p] = [p \ell_a].
\end{equation}   
(see ~\cite{davison} for a proof).
For our particular quiver $Q$, this algebra is non-commutative, but its center is described by a system of equations which cut out a cone on a del Pezzo surface.  
Our height change, as defined, is designed so that any two paths which are equal in $\bbC Q / I_w$ have the same height change.  Indeed, consider the map
\begin{align*}
\Psi_a: C_Q / I_w &\rightarrow \widetilde{Q} \times \bbZ_{\geq 0} 
\end{align*}
which sends a path to its endpoint in the universal cover $\widetilde{Q}$ together with its height change.  We have chosen the definition of the height function $h$ in such a way that $\Psi_a$ is well-defined, and in nice cases, namely when $Q$ is \emph{consistent},  $\Psi_a$ is injective.  Our quiver is in fact consistent. See ~\cite{bocklandt} for a number of equivalent definitions and implications of consistency, as well as~\cite{broomhead, davison,szendroi} for an introduction to the portions of the representation theory of quivers relevant to this project.  See also ~\cite{eager} for a physics perspective.

\section{The Generating Function}
\label{sec:generating_functions}

We aim to give a generating function which counts, for each perfect matching, the number of tiles needed to construct $M$.  Indeed, we can do slightly better: we can record some information about the orientations of these tiles.

\begin{theorem}

\begin{align*}
Z_{n}(a,b,c) &=\prod_{i = 0}^{n-1} (1 + a^ib^{i}c^{i+1})^{n-i}(1+a^ib^{i +1}c^{i+1 })^{n-i}. \\
Z_{n+\frac{1}{2}}(a,b,c) &=  \prod_{i=0}^{n} \left(1+a^{i+1}b^ic^i\right)^{n-i+1}
\prod_{i=0}^{n-1} 
\left(1+a^{i+1}b^{i+1}c^i\right)^{n-i}.
\end{align*}

\end{theorem}

\begin{proof}

Our strategy is to use a standard, folklore technique to encoding the weight of a perfect matching on $D_n$ as an edge weight function
\[
w_{n;abc} : E(D_n) \rightarrow \{\text{Laurent monomials in }a,b,c\}
\]
where a Laurent monomial is a term of the form $a^ib^jc^k$, $i,j,k \in \bbZ$.  This was worked out explicitly for the square lattice in~\cite{young-2009}.

Consider a hexagon-shaped fundamental domain $H$ of the $dP_3$ lattice, which has six squares, one square of each orientation.  Assign weight 1 to both of the horizontal edges in $H$, as well as to the eight edges along the upper and lower boundary of $H$.  Assign weight $X$ to the easternmost vertical edge of the SE kite in $H$,  and weight $Y$ to the westernmost vertical edge of the NW kite in $H$.  Having done this, there is now a unique way to assign Laurent monomial weights to the other edges in $H$ such that a plaquette flip around one of the six faces $f$ multiplies the weight of a perfect matching by the weight of $f$; it is shown in Figure~\ref{fig:weights around H} on the left.  Denote these weights by 
 $w^{\text{local}}_{X,Y;a,b,c}$.

\begin{figure}
\caption{Weights on the lattice, before and after shuffling.
\label{fig:weights around H}
\label{fig:weights around H afterwards}}
\begin{center}
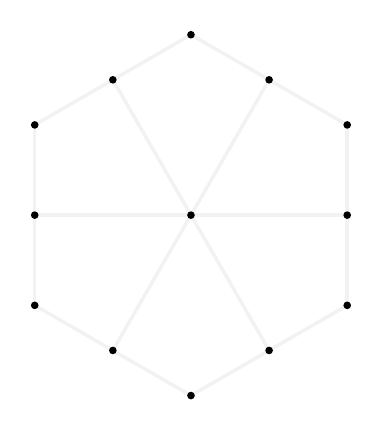
\rule{0.5in}{0in}
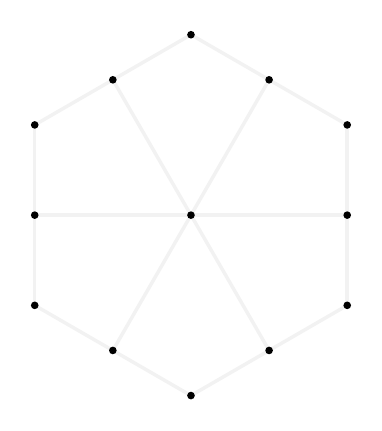
\end{center}
\end{figure}

For $n \in \frac{1}{2}\bbZ$, let $G$ be a minimal covering of $D_n$ by hexagons.  Define $w_{n;abc}(e) = 1$ for all vertical edges $e$ on the western boundary of $G$.  This has the effect of setting $Y=X(abc)^{-1}=1$ in $w^{\text{local}}_{X,Y;abc}$; for edges $e$ in these hexagons, let 
\[w_{n;a,b,c}(e) = w^{\text{local}}_{abc,1;a,b,c}.\]
We can now define $w_{n;a,b,c}$ inductively for the remaining edges in $G$.  Finally we restrict $w_{n;a,b,c}$ to $D_n$ to obtain the desired weight function. 

Then we have, for the NE-SW shuffle (which takes $Z_n$ to $Z_{n+\frac{1}{2}}$),
\begin{align*}
Z_n(a,b,c) 
&= K \! \sum_{\pi \text{ perfect matching}} \! w_{n;a,b,c}(\pi) \quad \in \bbZ[a,b,c]\\
\end{align*}

Now, when  we perform domino shuffling on this weighted lattice, we allow each edge to retain its weight (see Figure~\ref{fig:weights around H afterwards}).  This can be achieved by a simple change of variables in the generating function:
\begin{align*}
Z_{n+\frac{1}{2}}(a, (ac)^{-1}, (ab)^{-1}) &=
(1+a)^{n+1} \cdot K Z_n(a,b,c)
\end{align*}

Here, $K$ is a Laurent monomial such that the constant term of $Z_n(a,b,c) =1.$  It is the inverse of the edge weight assigned to the minimum height function.

Changing variables $x=a, y=(ac)^{-1}, z=(ab)^{-1}$ for the moment, and doing the similar computation for the $NW-SE$ shuffle, we obtain the following straightforward system of recurrences:

\begin{align*}
Z_{n+\frac{1}{2}}(x,y,z) &= K' \cdot (1+x)^{n+1} \cdot Z_n(x, (zx)^{-1}, (yx)^{-1}) \\
Z_{n+1}(x,y,z) &= K'' \cdot (1+z)^{n+1} Z_{n+\frac{1}{2}}((yz)^{-1}, (xz)^{-1}, z)
\end{align*}
for some Laurent monomials $K', K''$.

Let us solve these for the integer-order diamonds, leaving the half-integer order ones as a nearly identical exercise.  We change variables again: $x=a, y=b, z=c$.  It is also helpful to set $q=abc$, so that we obtain

\begin{align*}
Z_{n+1}(a,b,c) &= K'K'' (1+c)^{n+1}  \left(1+(bc)^{-1}\right)^{n+1} Z_n((bc)^{-1}, b, abc^2) \\
&= K'K''(1+c)^{n+1}\left(1+\frac{a}{q}\right)^{n+1} \cdot Z_n\left(\frac{a}{q}, b, qc\right).
\end{align*}

thus, factoring out $(\frac{q}{a})^{n+1}$, we have
\begin{align}
\label{eqn:recurrence_with_constant}
Z_{n+1}(a,b,c) &= K' \cdot K'' \cdot (1+c)^{n+1}\left(\frac{q}{a}\right)^{n+1}\left(1+\frac{q}{a}\right)^{n+1}Z_n\left(\frac{a}{q}, b, qc\right).
\end{align}

We know that the following are polynomials with constant term 1:
\[
\begin{array}{cccc}
Z_{n+1}(a,b,c), &
(1+c)^{n+1}, &
\left(1+\frac{q}{a}\right)^{n+1}, &
Z_n\left(\frac{a}{q}, b, qc\right).
\end{array}
\]

Hence, the monomial $K'K''\left(\frac{a}{q}\right)^{n+1}$ is equal to  1 and can be omitted from \eqref{eqn:recurrence_with_constant}, giving the easy linear recurrence relation
\[
Z_{n+1}(a,b,c) = (1+c)^{n+1}\left(1+\frac{q}{a}\right)^{n+1}Z_n\left(\frac{a}{q}, b, qc\right),
\]
It is elementary to check that the first formula in the theorem statement satisfies this first-order recurrence.  As such, the theorem will follow once we check that $Z_1(a,b,c)$ is equal to $(1+c)(1+bc)$, which it is.

\end{proof}

\section{Conclusion}

We should reiterate that our main contribution is the domino shuffling algorithm itself, rather than our enumeration formulae.  One other application of this algorithm is to a different type of boundary conditions on this lattice, coming from Donaldson-Thomas theory, analogous to the pyramid partitions studied in~\cite{young-2009}; this example was mentioned in~\cite{mozgovoy-reineke}.
Indeed, we were originally inspired to study this problem because of its algebro-geometric connections.  Accordingly, in a future paper, we will explain how to
compute this generating function.

The reader may have noted that when $n$ is an integer, the number of our order-$n$ diamonds is the square of the number of order $n$ Aztec diamonds; one might hope for a bijective proof of this fact.  However, this fact is \emph{not} true of the generating functions, so such a proof will likely be very difficult to find.

\begin{figure}[h]
\caption{A uniform random matching on an order 100 diamond 
\label{fig:bigguy}}
\begin{center}
\includegraphics[height=2.62in]{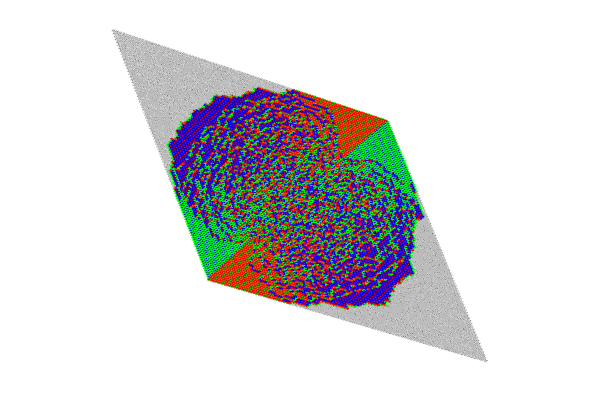}
\end{center}
\end{figure}

We have used our shuffling algorithm to generate several perfect matchings of large diamonds, chosen uniformly at random from the set of all matchings. An order-100 is shown in Figure~\ref{fig:bigguy}.  One can clearly see the boundary approaching a smooth curve; indeed it is natural~\cite{kenyon-okounkov} to guess that it is an algebraic curve, with boundary fluctuations given by the Airy kernel~\cite{johansson}.  The simplicity of the domino shuffle that we have described suggests that the approach of ~\cite{jockusch-propp-shor} may be effective in calculating this limiting curve: the frozen boundary evolves according to a TASEP-like process which is likely amenable to analysis. 

The first author was supported in part by an undergraduate research scholarship
from the Institut des Sciences Math\'ematiques (ISM).

\bibliographystyle{plain}
\bibliography{DP3}
\end{document}